\documentclass[12pt,leqno]{amsart}

\usepackage{hyperref}
\usepackage{mathrsfs}
\usepackage{tikz-cd}
\usepackage{amsmath,amsfonts,amssymb}
\usepackage{xfrac}
\usepackage{nicefrac}
\usepackage{color}
\usepackage[applemac]{inputenc}
\usepackage{anysize}
\usepackage{float}
\usepackage{amsmath}
\usepackage{amssymb}
\usepackage{amsbsy}
\usepackage{amsfonts}
\usepackage{amsthm}
\usepackage{latexsym}
\usepackage{color}

\usepackage[margin=1in]{geometry}

\newcommand{\Irr}{\operatorname{Irr}}

\newtheorem{lemma}{Lemma}[section]

\newtheorem{theorem}[lemma]{Theorem}

\newtheorem*{theoremA}{Theorem A}
\newtheorem*{theoremB}{Theorem B}
\newtheorem*{theoremC}{Theorem C}
\newtheorem*{theoremD}{Corollary D}

\newtheorem*{conjecture}{Conjecture}

\theoremstyle{definition}
\newtheorem{definition}{Definition}[section]
\newtheorem{remark}[definition]{Remark}
\newtheorem{corollary}[lemma]{Corollary}

\title{Word problems for finite nilpotent groups}

 \author[R.\,D.~Camina]{Rachel D. Camina}
 
 \address{Rachel D. Camina: Fitzwilliam College, Cambridge, CB3 0DG, UK}
 \email{rdc26@dpmms.cam.ac.uk}
 
 \author[A. I\~{n}iguez]{Ainhoa I\~{n}iguez}
 
\address{Ainhoa I\~{n}iguez: University of Mondragon, Faculty of Gastronomic Sciences, Donostia-San Sebastian, Spain}
\email{ainiguez@bculinary.com}
 
 \author[A. Thillaisundaram]{Anitha Thillaisundaram}
 
 \address{Anitha Thillaisundaram: School of Mathematics and Physics, University of Lincoln,
 	Brayford Pool, Lincoln LN6 7TS, UK}
 \email{anitha.t@cantab.net}


\begin{document}

\thispagestyle{empty}


\maketitle

\begin{abstract}
 Let $w$ be a word in $k$ variables. For a finite nilpotent group~$G$, a conjecture of Amit states that $N_w(1)\ge|G|^{k-1}$, where $N_w(1)$ is the number of $k$-tuples $\mbox{$(g_1,\ldots,g_k)\in G^{(k)}$}$ such that $w(g_1,\ldots,g_k)=1$. Currently, this conjecture is known to be true for groups of nilpotency class $2$. Here we consider a generalized version of Amit's conjecture, and prove that $N_w(g)\ge |G|^{k-2}$, where $g$ is a $w$-value in~$G$, for finite groups~$G$ of odd order and nilpotency class~$2$. 
 If $w$ is a word in two variables, we further show that $N_w(g)\ge |G|$, where $g$ is a $w$-value in~$G$ for finite groups~$G$ of nilpotency class~$2$.
 In addition, for $p$ a prime, we show that finite $p$-groups~$G$,  with two distinct irreducible complex character degrees, satisfy the generalized Amit conjecture for words $w_k=[x_1,y_1]\cdots[x_k,y_k]$ with $k$ a natural number; that is, for $g$ a $w_k$-value in~$G$ we have $N_{w_k}(g)\ge |G|^{2k-1}$.
 
 Finally, we discuss the related group properties of being rational and chiral, and show that every finite group of nilpotency class 2 is rational.
 \end{abstract}

\section{Introduction}

A word $w$ in $k$ variables $x_1,\ldots,x_k$ is an element in the free group~$F_k$ on $x_1,\ldots,x_k$. For any $k$ elements $g_1,\ldots,g_k$ in a group~$G$, we can define the element $w(g_1,\ldots,g_k)\in G$ by applying to $w$ the group homomorphism from $F_k$ to $G$ sending $x_i$ to $g_i$, for $1\le i\le k$.
 
 We denote by $G_w$ the set of word values of $w$ in~$G$, i.e. the set of elements $g\in G$ such that the equation $w=g$ has a solution in $G^{(k)}$, the direct product of $k$ copies of~$G$.
 
For a word $w$ in $k$ variables and a group $G$, for any $g\in G$ the \textit{fibre} of $g$ in $G^{(k)}$ is 
\begin{equation*}\label{fibre}
\{(g_1,\ldots, g_k)\in G^{(k)}\,|\,w(g_1,\ldots, g_k)=g\}.
\end{equation*}
 If $G$ is a finite group and $g\in G$, we define $N_{w,G}(g)$ to be 
\begin{equation*}\label{count sol}
N_{w,G}(g)=|\{(g_1,\ldots, g_k)\in G^{(k)}\,|\,w(g_1,\ldots, g_k)=g\}|;
\end{equation*}
i.e. the size of the fibre of $g$ in $G^{(k)}$. When the group~$G$ is clear, we will simply write $N_w(g)$. The function $N_{w,G}$ is a (non-negative) integer-valued class function since it is constant on the conjugacy classes. The set $\Irr(G)$, of complex irreducible characters of~$G$, is an orthonormal basis for the vector space of the complex class functions and $N_{w,G}$ can be written as a linear combination of the irreducible characters of~$G$:
 \begin{equation*}\label{N_w class function }
N_{w,G}=N_w=\sum_{\chi\in \Irr(G)}N_w^{\chi}\,\chi
\end{equation*}
where 
\begin{equation*}\label{coefficients}
N_{w}^{\chi}=(N_w,\chi)=\frac{1}{|G|}\sum_{g\in G}N_w(g)\overline{\chi(g)}=\frac{1}{|G|}\sum_{(g_1,\ldots,g_k)\in G^{(k)}}\overline{ \chi\big(w(g_1,\ldots,g_k)\big)}
\end{equation*}
is unique for any $\chi\in \Irr(G)$.
 
 Much about the functions $N_{w,G}$, or rather $P_{w,G}=N_{w,G}/|G|^k$, has been
done, particularly for the commutator word $w=[x,y]$ and for the case $G$ is a
$p$-group for some prime $p$; see~\cite{ErdosTuran, Gustafson, Rusin, Abert, carolyn, ainhoa-paper} and also~\cite{GarionShalev, AmitVishne, DasNath, NathYadav, AntolinMartinoVentura, DJMN}.
In addition, Nikolov and Segal~\cite{NikolovSegal}
gave a characterization of finite nilpotent groups and of finite solvable
groups based on the function~$P_{w,G}$: a finite group is nilpotent if and only if the
values of $P_{w,G}(g)$ are bounded away from zero as $g$ ranges over $G_w$ and
$w$ ranges over all group words; and a finite group is solvable if and only if
the probabilities $P_{w,G}(1)$ are bounded away from zero as $w$ ranges over
all group words. I\~{n}iguez and Sangroniz~\cite{ainhoa-paper} proved that for any finite group $G$ of nilpotency class $2$ and any word $w$, the function $N_w$ is a generalized character of $G$,  that is, a $\mathbb{Z}$-linear combination of irreducible characters. What is more, if $G$ is a finite $p$-group of nilpotency class $2$ with $p$ odd and $w$ any word, then $N_w$ is a character of $G$. In general, for $p=2$ the function  $N_w$ is not a character; one can easily check for $N_{x^2,Q_8}$. In~\cite{ainhoa-paper} the authors also characterize when the function $N_{x^n}$ is a character for $2$-groups of nilpotency class $2$.

 The following is well known; see~\cite{Abert}:
 
\begin{conjecture} (Amit) 
For any finite nilpotent group~$G$ and any word~$w$ in $k$ variables, 
\[
N_w(1)\ge |G|^{k-1}.
\]
\end{conjecture}
 
 Up till now, Amit's conjecture has only been proved for groups of nilpotency class $2$. This was done by Levy~\cite{Levy} and independently by I\~{n}iguez and Sangroniz~\cite{ainhoa-paper}.

 Amit's conjecture is seen to hold for certain words~$w$. If $w$ is a  two-variable word, then $N_w(1)\ge |G|$ for all finite nilpotent groups~$G$, by Solomon~\cite{Solomon}. Whenever $N_{w,G}$ is a character, Amit's
conjecture also holds; see~\cite{ainhoa-paper}. It then follows that Amit's conjecture holds for all left-normed commutators $w_n=[x_1,\ldots, x_n]$  and for all generalized commutators $v_n=x_1x_2\cdots x_nx_1^{-1}x_2^{-1}\cdots
x_n^{-1}$; see~\cite{PrajapatiNath} and~\cite{Tambour} respectively.

We consider a version of Amit's conjecture as applied to general 
fibres:
\begin{conjecture}(Generalized Amit Conjecture)
For any finite nilpotent group~$G$, any word $w$ in $k$ variables and any $g\in G_w$,
\[
N_w(g)\ge |G|^{k-1}.
\]
\end{conjecture}
This appears as a conjecture in Ashurst's thesis \cite[Conjecture 6.2.1]{carolyn}. 
Note that the bound $N_w(g)=|G|^{k-1}$ is achieved by a surjective word map with uniform distribution, for example $w=x[y,z]$. Moreover, Cocke and Ho have shown that a finite group is nilpotent if and only if every surjective word map has uniform distribution~\cite[Theorem~B]{CockeHo2}, so the Amit bound is met by all surjective word maps.
By Solomon's result in~\cite{Solomon} we also know that if $w$ is a two variable word, then $N_w(g)\ge |G|$ for all $g \in Z(G)$ and all finite nilpotent groups~$G$. Here we improve this result to all groups $G$ of nilpotency class $2$ and all $g \in G_w$.

However first note that  since a finite nilpotent group is a direct product of its Sylow subgroups, it suffices to consider finite $p$-groups. This is because if $G=H\times K$, and $g=hk\in G_w$ for an $n$-variable word~$w$ with $h\in H$ and $k\in K$, then $\mbox{$N_{w,G}(g) = N_{w,H}(h) N_{w,K}(k)$}$. This relies on the fact that if $g_i=h_i k_i$ for $1\le i\le n$ with $h_i\in H$ and $k_i\in K$, then $w(g_1,\ldots,g_k)=w(h_1k_1,\ldots,h_nk_n)=w(h_1,\ldots, h_n)w(k_1,\ldots, k_n)$. Hence if the conjecture holds for each group $H$ and $K$, it then holds for their direct product $G$. 

\begin{theoremA} \label{thmA} Suppose $G$ is a finite $p$-group of nilpotency class $2$ and $w$ is a word in two variables. Then
$N_w(g) \geq |G|$ for all $g \in G_{w}$.
\end{theoremA}

 I\~{n}iguez and Sangroniz~\cite{ainhoa-paper} proved that the generalized Amit
conjecture holds for free $\mbox{$p$-groups}$ of nilpotency class $2$ and exponent $p$. Our next result does not meet Amit's bound, but can be proved for all words $w$ and for all groups $G$ of odd order and nilpotency class $2$.

\begin{theoremB} \label{thmB} Suppose $G$ is a finite $p$-group of nilpotency class $2$, for $p$  an odd prime,  and $w$ is a word in $k$ variables. Then $N_w(g) \geq |G|^{k-2}$ for all $g \in G_w$.
\end{theoremB}

Next, we extend a result of Pournaki and Sobhani to
words $w_k$ of the form\linebreak 
$\mbox{$w_k=[x_1, y_1] \cdots [x_k,y_k]$}$. Pournaki
and Sobhani originally considered the single commutator $w_1$ \cite[Theorem~2.2]{pournaki}.  
Before stating our result, we  recall that ${\rm cd}(G)$ denotes
the set of degrees of irreducible complex characters of $G$, and  ${\rm cs}(G)$ denotes the set of conjugacy class sizes in~$G$.

\begin{theoremC} \label{thmC} Let $G$ be a finite $p$-group such that 
${\rm cd}(G) = \{1,m\}$
for $m >1$ and $\mbox{$w_k=[x_1, y_1] \cdots [x_k,y_k]$}$ a product of $k$
distinct commutators, for $k\in\mathbb{N}$. Then $G'=G_{w_k}$ and  $\mbox{$|\{N_{w_k}(g): g \in G'\}|=2$}$. Furthermore $N_{w_k}(g) \geq |G|^{2k-1}$ for all $g\in G'$.
\end{theoremC}

 It is interesting to note, that if instead of requiring $|{\rm cd}(G)| =2$
we require $|{\rm cs}(G)| =2$ then there exist groups with
$|\{N_{w_k}(g): g \in G_{w_k}\}|=n$ for all positive integers $n$. This is a
recent result due to Naik \cite{naik}.

Theorem C yields the following corollary.

\begin{theoremD} Let $G$ be a finite group of nilpotency class $2$ and 
$|G'|=p$, with $p$ a prime. Then $N_{w_k}(g) \geq |G|^{2k-1}$ whenever
$g \in G'$.
\end{theoremD}

We remark that the same result for $N_{w_k}$ was obtained in~\cite[Propositions 6.1 and 6.2]{ainhoa-paper} for finite $p$-groups with different restrictions.

Finally we consider the different notions of rationality and chirality; see Section~\ref{sec:rational} for definitions. In
particular, we point out that if $G$ is a finite group of nilpotency class $2$ and
$w$ is a word then $N_w(g) = N_w(g^e)$ for all $e$ coprime to the order
of $G$. This is an improvement on~\cite[Theorem 5.2]{cocke} for the case of finite groups.

\smallskip

All groups in this paper are finite.

\medskip

\textbf{Acknowledgements.}  We thank Dan Segal for many useful conversations. The first author would like to thank the Isaac Newton Trust for supporting her sabbatical during which some of this research took place. The second and third authors thank the Department of Pure Mathematics and Mathematical Statistics of the University of Cambridge, and the third author also thanks Fitzwilliam and Homerton Colleges,  for supporting several research visits to Cambridge.


\section{Fibres of non-identity elements}\label{sec:fibres}

Given a group $G$ of nilpotency class $2$ and a word $w$, we consider the sizes of fibres of 
non-identity elements under the word map. 

First we observe that if the word map
$w:G^{(k)} \rightarrow G$ which sends $(x_1, \ldots, x_k)$ to 
$w(x_1, \ldots, x_k)$ is a homomorphism then the fibre of any element in $G_w$ 
is a coset of the kernel of the map. Hence the fibre of each element in $G_w$
is of the same size, namely $|G|^k/|G_w|$ (which is at least $|G|^{k-1}$).
When $G$ is abelian all word maps are homomorphisms. We use this idea
to analyse the nilpotency class $2$ case.

Another key observation, that will be used throughout, is that 
$[xy,z] = [x,z][y,z]$ and $[x,yz]=[x,y][x,z]$ in a group of nilpotency class $2$.

Two words $w, w' \in F_k$ are said to be {\it equivalent} if they belong
to the same orbit under the action of the automorphism group of $F_k$.
In \cite[Proposition~2.1]{ainhoa-paper} the authors prove if $w$ and $w'$ are
equivalent then $N_{w,G} = N_{w',G}$ for any finite $p$-group $G$ of
nilpotency class $2$. They then go on to prove that the following words
are a system of representatives of the action of ${\rm Aut}(F_k)$ on
$F_k$ \cite[Proposition~2.3]{ainhoa-paper}:

\begin{equation}\label{eq:1}
[x_1, x_2]^{p^{s_1}} \cdots [x_{2r-1}, x_{2r}]^{p^{s_r}}, \quad\text{for }0 \leq s_1 \leq \cdots \leq s_r,
\end{equation}
\begin{equation}\label{eq:2}
x_1^{p^{s_1}} [x_1, x_2]^{p^{s_2}} [x_2, x_3]^{p^{s_3}} \cdots 
[x_{r-1}, x_r]^{p^{s_r}}, \quad\text{for } s_1 \geq 0,\; 0 \leq s_2 \leq \cdots \leq s_r .
\end{equation}

Thus, it is enough for us to consider words of these types. We can now prove Theorem~A.

\begin{proof}[Proof of~Theorem~A] 
Let $Z$ denote the centre of~$G$. We first consider words of type~(\ref{eq:1}), so 
$w = [x_1, x_2]^{p^{s_1}}$. In this case $G_w \subseteq Z$ and the result follows from Solomon's result~\cite{Solomon}.

Now consider words of type~(\ref{eq:2}), so $w = x_1^{p^{s_1}}[x_1, x_2]^{p^{s_2}}$. 
If $G^{p^{s_1}} \leq Z$ then again the result 
follows from~\cite{Solomon}. So, suppose $G^{p^{s_1}}$ is not central, then 
$Z^{p^{s_1}} \neq 1$. We now proceed by induction on the order of $G$, noting that the result holds for abelian groups.

Suppose $g \in G_w$ and
$\mathbf{\Omega} = w^{-1}(g)$, the preimage of $g$ in $G^{(2)}$. Let $N = Z^{p^{s_1}}$
and consider $\bar{G} = G/N$. Set 
$\bar{\mathbf{\Omega}} = w^{-1}(\bar{g}) \subseteq \bar{G}^{(2)}$. Inductively
$|\bar{\mathbf{\Omega}}| \geq |\bar{G}|$. For each ${\bf v} \in \bar{\mathbf{\Omega}}$
choose a representative  $(a_1, a_2) \in G^{(2)}$ with
$(\bar{a}_1, \bar{a}_2) = {\bf v}$. Then
$w(a_1, a_2) = gu^{p^{s_1}}$ for some $u \in Z$ and then
$w(a_1u^{-1}, a_2 s_2) = g$ for all $s_2 \in N$. So
$$
\mathbf{\Omega} \supseteq \bigcup _{{\bf v} \in \bar{\mathbf{\Omega}}} \{a_1 u^{-1}\} \times
a_2 N
$$
a disjoint union. Hence $|\mathbf{\Omega}| \geq |N| |\bar{\mathbf{\Omega}} |\geq |G|$.
\end{proof}

Before proving
Theorem~B we introduce one more concept, that of the `defined word map'.
We are used to a word $w \in F_k$ defining a word map from $G^{(k)}$ to $G$.
In a defined word map, some of the entries are fixed elements of
$G$ and are not allowed to vary. For a fixed tuple $(a_1, ..., a_k)\in G^{(k)}$, we will write $w^{(i_1, ..., i_k)}_{(a_1, ..., a_k)}$ for the defined word map where the $ i_j$-th term is replaced with $a_j$. This is particularly useful when $G$ is of nilpotency class $2$ as then 
this defined word map is often a homomorphism.

\begin{proof}[Proof of~Theorem~B] We first consider words of type~(\ref{eq:1}), in this case the argument also works for $p=2$. 
Given $w$ of type~(\ref{eq:1}) we consider
the corresponding defined word map given by fixing the even entries as
$(a_2, a_4, \ldots, a_{2r}) \in G^{(r)}$ say. That is
$$
w^{(2,4,\ldots, 2r)}_{(a_2, a_4, \ldots, a_{2r})}: G \times \overset{r}\cdots \times G
\rightarrow G
$$
$$(x_1, x_3, \ldots, x_{2r-1}) \mapsto [x_1, a_2]^{p^{s_1}} \cdots 
[x_{2r-1}, a_{2r}]^{p^{s_r}}.$$ As $G$ is of nilpotency class $2$ this map is a 
homomorphism. Furthermore, the image is a
subgroup of the centre~$Z$ of~$G$ and thus the kernel has size at
least $|G|^{r}/|Z|$.

Suppose $g \in G_w$ and in particular 
$g = w(a_1, a_2, a_3 \ldots, a_{2r})$ for some 
$a_i \in G$. Now, by the previous paragraph, if we fix the $a_i$ for $i$ even, 
we see there are at 
least $|G|^{r}/|Z|$ tuples
$(b_1, b_3, \ldots, b_{2r-1})$ satisfying 
$g = w(b_1, a_2, b_3, \ldots, b_{2r-1}, a_{2r})$.  
Fix such a  $(b_1, b_3, \ldots, b_{2r-1})$, and construct
the new defined word map
$w^{(1,3,\ldots, 2r-1)}_{(b_1, b_3, \ldots, b_{2r-1})}$ which sends $G^{(r)}$ to $G$ by mapping
$(x_2, x_4, \ldots, x_{2r})$ to $w(b_1, x_2, \ldots, b_{2r-1}, x_{2r})$. Again
this is a homomorphism to $Z$. Note
that $g$ lies in the image of each of these maps, and the preimage of $g$
for each of these homomorphisms has size at least $|G|^r/|Z|$.
Thus summing over the $r$-tuples $(b_1, b_3, \ldots, b_{2r-1})$ yields
that the $N_w(g) \geq (|G|^r/|Z|)(|G|^r/|Z|) = |G|^{2r}/|Z|^2 \geq |G|^{2r-2}$.

We now consider words of type~(\ref{eq:2}). We consider different cases 
depending on the triviality, or otherwise, of $Z^{p^{s_1}}$.
When $Z^{p^{s_1}} \neq 1$ 
we use induction on
the order of the group: we assume for all groups of smaller order of 
nilpotency class $2$ or less our result holds. 
Note the result holds for abelian groups so we have the base step.\\[1ex]
{\it Case(i)}: Suppose $Z^{p^{s_1}}=1$.\\ 
Suppose $g \in G_w$ and in particular $g = w(a_1, \ldots, a_r)$ for some
$a_i \in G$. Define
$$
\bar{w}=[x_1, x_2]^{p^{s_2}}[x_2, x_3]^{p^{s_3}} \cdots 
[x_{r-1}, x_r]^{p^{s_r}}.
$$ 
Then $a_1^{-p^{s_1}}g = \bar{w}(a_1, \ldots, a_r)$.  
Fixing the odd elements and constructing the corresponding defined word map
$\bar{w}^{(1,3,\ldots, t)}_{(a_1, a_3, \ldots, a_t)}$, where $t = 2\lceil r/2 \rceil -1$, 
gives a homomorphism into $Z$. Thus the number of tuples 
$(b_2, \ldots, b_s)$, with $s=2 \lfloor r/2 \rfloor$, 
such that $\bar{w}(a_1, b_2, \ldots) = a_1^{-p^{s_1}}g$ and thus
$w(a_1, b_2, \ldots) =g$ is at least $|G|^{\lfloor r/2 \rfloor}/|Z|$.

Fixing the even elements of $w$ does define a homomorphism as we have
insisted $Z^{p^{s_1}}=1$ and $p$ odd and thus 
$(y_1 y_2)^{p^{s_1}} = y_1^{p^{s_1}}y_2^{p^{s_1}}$ as required. Furthermore
$G^{p^{s_1}}$ is central, as $\mbox{$Z^{p^{s_1}} = 1$}$, and thus the corresponding 
defined word map is
a homomorphism into $Z$. So, fixing the even elements yields a homomorphism
with kernel of order at least
$|G|^{\lceil r/2 \rceil}/|Z|$. Thus, combining as before, gives that
$g$ has a fibre of size at least 
$$
(|G|^{\lfloor r/2 \rfloor}/|Z|)(|G|^{\lceil r/2 \rceil}/|Z|)= 
|G|^r/|Z|^2 \geq |G|^{r-2}.
$$
When $Z^{p^{s_1}} \neq 1$ we proceed by induction and use the usual word
map, as seen below.\\[1ex]
{\it Case(ii)}: Suppose $Z^{p^{s_1}} \neq 1$.\\
Here we proceed analogously to the last paragraph in the proof of Theorem~A. 
In the notation of that proof, we obtain
$$ 
\mathbf{\Omega} \supseteq \bigcup _{{\bf v} \in \bar{\mathbf{\Omega}}} \{ a_1u^{-1}\} \times 
a_2N \times
\cdots \times a_rN,
$$
a disjoint union. Thus
$|\mathbf{\Omega}| \geq |N|^{r-1} |\bar{\mathbf{\Omega}}| \geq |G|^{r-2}$.
\end{proof}

\begin{remark}
(i) For a word~$w$ of type~(\ref{eq:1}) and all primes $p$, if $|Z|^2\le |G|$ then the generalized Amit conjecture holds. 

\noindent(ii) For a word~$w$ of type~(\ref{eq:2}) with $s_1=0$, the word map defined by~$w$ is surjective and hence its distribution is uniform; cf.~\cite[Lemma~3.2.1]{carolyn} or~\cite[Theorem~B]{CockeHo2}. This implies that the generalized Amit conjecture holds. 
\end{remark}

\section{Characters}\label{sec:characters}

Here we show, using character theory techniques, that the generalized Amit conjecture holds for certain words and certain groups .

Recall $w_k(x_1, y_1, \ldots, x_k, y_k) = [x_1, y_1] \cdots [x_k, y_k]$ is
the product of $k$ disjoint commutators for $k\in\mathbb{N}$, and ${\rm cd}(G)$ is
the set of degrees of irreducible complex characters of $G$. 
The following results first appeared in the second author's thesis. 

\begin{theorem}\label{thm:non-trivial-g} Let $G$ be a finite $p$-group
such that ${\rm cd}(G) = \{1,m\}$ for $m >1$. If $1 \neq g \in G'$ 
 then
\[
N_{w_k}(g) = \frac{|G|^{2k}}{|G'|} \Big(1 - \frac{1}{m^{2k}}\Big).
\]
Furthermore  $G'=G_{w_k}$ and $N_{w_k}(g) \geq |G|^{2k-1}$ for $1\ne g\in G'$.
\end{theorem}

\begin{proof}
For $1\neq g\in G'$, using the second orthogonality relation~\cite[Theorem~2.18]{isaacs}, we have
\begin{align*}
0&= \sum_{\chi\in \text{Irr}(G)}\chi(g)\chi(1)=\sum_{{
\chi \in\text{Irr}(G)}\atop{\chi(1)=1}} \chi(g) +\sum_{{
\chi \in\text{Irr}(G)}\atop{\chi(1)=m}} \chi(g)\chi(1) \\
&=\sum_{{
\chi \in\text{Irr}(G)}\atop{\chi(1)=1}} \chi(g) +m\sum_{{
\chi \in\text{Irr}(G)}\atop{\chi(1)=m}} \chi(g).
\end{align*}
Now, noting from~\cite[Corollary~2.23]{isaacs} that the number of irreducible linear characters is $|G:G'|$, and from~\cite[Lemma~2.19]{isaacs} that if $\chi$ is linear, then $\chi(g)=\chi(1)$ since $G'\le \ker \chi$, we obtain
\[
0=|G:G'| + m\sum_{{
\chi \in\text{Irr}(G)}\atop{\chi(1)=m}} \chi(g).
\]
Therefore,
\[
\sum_{{
\chi \in\text{Irr}(G)}\atop{\chi(1)=m}} \chi(g)=-\frac{|G:G'|}{m}.
\]
Next, as $N_{w_k}^{\chi}=(N_{w_k},\chi)=\big(\frac{|G|}{\chi(1)}\big)^{2k-1}$ for any $\chi\in\text{Irr}(G)$ by~\cite[Theorem~1]{Tambour}, we have
\begin{align*}
N_{w_k}(g)&=\sum_{\chi \in \text{Irr}(G)}\Big(\frac{|G|}{\chi(1)}\Big)^{2k-1}\cdot \chi(g)\\
&=\sum_{{
\chi \in\text{Irr}(G)}\atop{\chi(1)=1}} |G|^{2k-1} +\sum_{{
\chi \in\text{Irr}(G)}\atop{\chi(1)=m}}  \Big(\frac{|G|}{m}\Big)^{2k-1}\cdot \chi(g)\\
&=|G|^{2k-1}\cdot |G:G'| +\Big(\frac{|G|}{m}\Big)^{2k-1}\cdot \Big(\frac{-|G:G'|}{m}\Big)\\
&=\frac{|G|^{2k}}{|G'|}\cdot \Big(1-\frac{1}{m^{2k}}\Big),
\end{align*}
hence the first result. For the final statement, we note that since $m\ge 2$, we have
\[
 \Big(1-\frac{1}{m^{2k}}\Big)\ge \frac{3}{4}~.
\]
Consequently, all elements in $G'$ appear as images of $w_k$; so $G'=G_{w_k}$. What is more, since~$G$ is non-abelian, we have $|G:G'|\ge 2$ and hence the lower bound is proved for the fibres.
\end{proof}

\begin{proof}[Proof of Theorem~C]
By Theorem~\ref{thm:non-trivial-g}, it remains to prove that $N_{w_k}(1)\ge |G|^{2k-1}$ and that $N_{w_k}(1)\ne N_{w_k}(g)$ for $1\ne g\in G'$. For the first part, note that 
\begin{equation}\label{eq:full-sum}
N_{w_k}(1)+\sum_{1\ne g\in G'} N_{w_k}(g) =\sum_{g\in G'} N_{w_k}(g) = |G|^{2k}.
\end{equation}
Since we showed in Theorem~\ref{thm:non-trivial-g} that $G'=G_{w_k}$, it follows from the previous result that for $1\ne g\in G'$,
\[
N_{w_k}(g) = \frac{|G|^{2k}}{|G'|}\Big(1-\frac{1}{m^{2k}}\Big) < 
\frac{|G|^{2k}}{|G'|}
\]
and hence $N_{w_k}(1)>\frac{|G|^{2k}}{|G'|}> |G|^{2k-1}$ using~(\ref{eq:full-sum}). 

In particular, we note that $N_{w_k}(1)>\frac{|G|^{2k}}{|G'|}> N_{w_k}(g)$ for $1\ne g\in G'$, proving there exist exactly two fibre sizes.
\end{proof}

\begin{proof}[Proof of Corollary~D]
For $g=1$, the result is true by~\cite{ainhoa-paper} and~\cite{Levy}. We claim that a non-linear irreducible character~$\chi$ vanishes outside of the centre~$Z$ of~$G$. Consider $\mbox{$g\in G\backslash Z$}$. So there exists some $x\in G$ such that $t=[g,x]\neq 1$. Since $|G'|=p$, the element~$t$ is a generator of~$G'$. If we consider now a complex representation $\rho$ affording $\chi$, we have that $\rho(t)=\epsilon I$ where $\epsilon\in\mathbb{C}$ by~\cite[Lemma~2.25]{isaacs}. In the case $\epsilon=1$, we have $t\in\ker\rho$ and therefore $G'\le \ker\rho$ which is a contradiction to $\chi$ being non-linear; compare~\cite[Lemma~2.22]{isaacs}. 
Therefore $\epsilon\neq 1$ and since 
\[
\chi(g)=\chi(g^{x})=\chi(gt)=\text{tr}_{\rho}(gt)=\text{tr}(\rho(g)\rho(t))=\text{tr}(\epsilon\rho(g)I)=\epsilon\,\chi(g),
\]
 concluding that $\chi(g)=0$, and the claim holds. From~\cite[Corollary~2.28 and Lemma~2.29]{isaacs}, we deduce that $\chi(1)^2=|G:Z|$. Therefore $G$ is a group of central type with just two irreducible complex character degrees, i.e. ${\rm cd}(G) = \{1,|G : Z|^{1/2}\}$. Now the assertion holds using the previous theorem.
\end{proof}

\section{Rationality and Chirality}\label{sec:rational}

In this section we draw together some definitions and ideas that have appeared
in the literature and conclude with a corollary which, although is a direct consequence of the results of~\cite{ ainhoa-paper}, has not previously been explicitly stated and we believe is of interest.

According to~\cite{cocke} a pair $(G,w)$, where $G$ is a group and 
$w$ is a word, is called \textit{chiral} if $G_w\neq G_w^{-1}$. The group $G$ is called \textit{chiral} if $(G,w)$ is chiral for some $w$. Otherwise $G$ is \textit{achiral}. In~\cite{cocke} the authors comment that the existence of chiral groups
follows from a result of Lubotzky~\cite{lub}. They then began the process
of classifying all finite chiral groups. In particular, they
found all chiral groups of order less than 108; there
are two of them. These results negatively answer a question posed by Ashurst in
her thesis~\cite[Question 5]{carolyn}:
If $G$ is a finite group, $g \in G$ and $w \in F_{\infty}$ is it necessarily
true that $P(G, w=g) = P(G, w=g^{-1})$?

Related to the definition of achiral is the definition of weakly rational.
According to~\cite{gur-15},
a word $w$ is \textit{rational} if for every finite group~$G$ and any 
$g\in G$, we have $N_w(g)=N_w(g^{e})$ for every $e$ relatively prime to~$|G|$. 
Additionally, a word $w$ is \textit{weakly rational} if and only if for every 
finite group~$G$ and for every integer $e$ relatively prime to~$|G|$, the set~$G_w$
 is closed under $e$-th powers. 
Clearly rational implies weakly rational; see~\cite{gur-15} for more 
discussion. 

We change the emphasis of the definition and say a pair $(G,w)$ for $G$ a group
and $w$ a word is 
\textit{rational} if for all $g \in G$ and for every $e$ relatively prime
to $|G|$, we have
$N_w(g)=N_w(g^e)$. A group $G$ is \textit{rational} if $(G,w)$ is 
rational for all words $w$. Similarly we define
a pair $(G,w)$ to be \textit{weakly rational} if for every $e$
relatively prime to $|G|$, the set $G_w$ is closed under $e$-th powers.
A group $G$ is \textit{weakly rational} if $(G,w)$ is weakly rational for all
pairs $(G,w)$ running over all words~$w$. Clearly if $G$ is rational it is weakly rational and
if it is weakly rational it is achiral.

In \cite{ainhoa-paper} the
authors show that the pair $(G,w)$ is rational if and only if $N_w$ is a 
generalized character of $G$.

\begin{lemma}\cite[Lemma 3.1]{ainhoa-paper}\label{rational words}
Let $G$ be a group and $w$ a word. Then $N_w=N_{w,G}$ is a generalized character of $G$ if and only if $N_w(g)=N_w(g^{e})$ for any $g\in G$ and $e$ relatively prime with the order of $G$. 
\end{lemma}

In particular they showed this is exactly what happens for any word $w$ and 
any finite group of nilpotency class $2$.
\begin{theorem}\cite[Theorem 3.2]{ainhoa-paper}\label{generalized}
Let $G$ be a $p$-group of nilpotency class $2$ and $w$ a word. 
Then $N_w=N_{w,G}$ is a generalized character of $G$. 
\end{theorem}
We include these results here to highlight
the following 
corollary which is a partial improvement on \cite[Theorem 5.2]{cocke} which says
that all class $2$, rank 3, nilpotent groups are achiral.

\begin{corollary}
Every finite group $G$ of nilpotency class $2$ is rational.
\end{corollary}

\begin{proof} First note an abelian
group is rational, and by Lemma~\ref{rational words} and Theorem~\ref{generalized} a $p$-group of
nilpotency class $2$ is rational. It follows for finite nilpotent groups of class $2$ using the comment before Theorem~A.
\end{proof}

\end{document}